\documentclass[a4paper,12pt,reqno]{amsart}
\usepackage{amssymb}
\usepackage{amsmath}
\usepackage{amsthm}
\usepackage{amsfonts}

\newtheorem{theorem}{Theorem}[section]
\newtheorem{lemma}[theorem]{Lemma}

\theoremstyle{definition}
\newtheorem{definition}[theorem]{Definition}

\theoremstyle{remark}

\numberwithin{equation}{section}

\begin{document}

\title[Orthonormality]{Orthonormality of wavelet system on the Heisenberg group and twisted wavelet system on $\mathbb{C}$}

\author{S.Arati}
\author{R.Radha*}
\address{Department of Mathematics, Indian Institute of Technology Madras, Chennai 600 036, India}
\email{aratishashi@gmail.com ; radharam@iitm.ac.in}
\thanks{* Corresponding author}



\begin{abstract}
The aim of this paper is to obtain necessary and sufficient conditions for the orthonormality of wavelet system arising out of left translations and nonisotropic dilations on the Heisenberg group $\mathbb{H}$. A similar problem is also discussed for the twisted wavelet system on $\mathbb{C}$.
\end{abstract}

\maketitle
\section{Introduction and background}
Let $\psi\in L^2(\mathbb{R})$. Define $\psi_{j,k}$ as $\psi_{j,k}(x)=2^{j/2}\psi(2^jx-k)$, for $j,k\in\mathbb{Z}$ and $x\in\mathbb{R}$. In other words, $\psi_{j,k}=D_{2^j}T_k\psi$, where $T_u,u\in\mathbb{R}$ denotes the translation operator $T_uf(x)=f(x-u)$ and $D_a,a\in\mathbb{R}^{\ast}$ denotes the dilation operator $D_af(x)=|a|^{1/2}f(ax),x\in\mathbb{R}$. This system $\{\psi_{j,k}:j,k\in\mathbb{Z}\}$ is in general called a wavelet system (cf \cite{Chris}). The following well known result gives a characterization of the orthonormality of the wavelet system.
\begin{theorem}
Let $\psi\in L^2(\mathbb{R})$. The necessary and sufficient conditions for the orthonormality of the system $\{\psi_{j,k}:j,k\in\mathbb{Z}\}$ are
\begin{align*}
&\sum_{k\in\mathbb{Z}}|\hat{\psi}(\xi+k)|^{2}=1\quad \text{for a.e.}\,\xi\in\mathbb{R}\\
\text{and}\quad&\sum_{k\in\mathbb{Z}}\hat{\psi}(\xi+k)\overline{\hat{\psi}(2^j(\xi+k))}=0\quad \text{for a.e.}\,\xi\in\mathbb{R},\,j\geq1.
\end{align*}
\end{theorem}
\noindent
For the proof, we refer to \cite{HW}.
\par In this note, we wish to consider a similar problem for the wavelet system on the Heisenberg group $\mathbb{H}$. This wavelet system emerges from the left translations $L_{(k,l,m)},(k,l,m)\in\mathbb{Z}^3$ and the nonisotropic dilations $\delta_{2^j},j\in\mathbb{Z}$ which are defined as follows. For $\psi\in L^2(\mathbb{H})$,
\begin{align*}
L_{(u,v,s)}\psi(x,y,t)&=\psi((u,v,s)^{-1}(x,y,t))\\
&=\psi\left(x-u,y-v,t-s+\frac{1}{2}(y\cdot u-x\cdot v)\right),\\
(u,v,s)\in\mathbb{R}^3\text{ and }
\delta_a&\psi(x,y,t)=|a|^2\psi(ax,ay,a^2t), \, a\in\mathbb{R}^\ast.
\end{align*}
The main result is considered in Theorem \ref{T:orthoHei}. While looking into this problem, we came across the wavelet system on $\mathbb{C}$ which appears due to twisted translations and dilations on $\mathbb{C}$. We call this system on $\mathbb{C}$ as twisted wavelet system. When we tried to study the necessary and sufficient condition for the orthonormality of twisted wavelet system it turned out to be a surprising fact that the required conditions became too complicated unlike the classical result on $\mathbb{R}$ or the result on the Heisenberg group. In fact, there are five necessary and sufficient conditions which lead to the orthonormality of the twisted wavelet system. This result is stated in Theorem \ref{T:orthoTwis}.
\par
Now, we shall mention a few works based on system of translates in various settings available in the literature. Characterizations of shift invariant spaces in $L^2(\mathbb{R}^n)$ in terms of range functions were obtained by Bownik in \cite{Bow}. These results were later extended to locally compact abelian groups in \cite{Cab} and \cite{Kamy}. For nonabelian compact groups, the shift invariant spaces were explored in \cite{RaSh}. Characterization of the orthonormality of a system of translates on the polarised Heisenberg group was studied in \cite{Azit} using the concept of bracket map. The shift invariant spaces associated with the twisted translations for $L^2(\mathbb{C}^n)$ were recently studied by Radha and Adhikari in \cite{RaS}. In \cite{RaS1}, characterizations of Bessel sequences, orthonormal bases, frames and Riesz bases were studied on the Heisenberg group for a shift invariant space with countably many mutually orthogonal generators. The structural properties of shift-modulation invariant spaces were studied by Bownik in \cite{Bow1}. These results were extended to locally compact abelian groups in \cite{Cab1}. 
\par
At this point, we shall provide the necessary background to understand our main results. The Heisenberg group $\mathbb{H}^n$ is a Lie group whose underlying manifold is $\mathbb{R}^n\times\mathbb{R}^n\times\mathbb{R}$ which satisfies the group law
\[(x,y,t)(u,v,s)=\left(x+u,y+v,t+s+\frac{1}{2}(u\cdot y-v\cdot x)\right).\]
It is a nonabelian noncompact locally compact group. The Haar measure on $\mathbb{H}^n$ is the Lebesgue measure $dxdydt$. It follows from the well known Stone-von Neumann theorem that every infinite dimensional irreducible unitary representation on the Heisenberg group is unitarily equivalent to the representation $\pi_\lambda,\,\lambda\in\mathbb{R}^\ast$, where $\pi_\lambda$ is defined by
\[\pi_\lambda(x,y,t)\varphi(\xi)=e^{2\pi i\lambda t}e^{2\pi i\lambda(x\cdot\xi+\frac{1}{2}x\cdot y)}\varphi(\xi+y),\quad\varphi\in L^{2}(\mathbb{R}^n).\]

For $f\in L^{1}(\mathbb{H}^n)$, the group Fourier transform $\hat{f}$ is defined as follows. For $\lambda\in\mathbb{R}^\ast,\,\hat{f}(\lambda)$ given by 
\[\hat{f}(\lambda)=\int\limits_{\mathbb{C}^n\times\mathbb{R}}f(z,t)\pi_\lambda(z,t)dzdt\] is a 
bounded operator on $L^{2}(\mathbb{R}^n)$. The inverse Fourier transform of $f\in L^{1}(\mathbb{H}^n)$ in the $t$ variable, denoted by $f^\lambda$, is defined as  
\begin{equation}
f^\lambda(z)=\int\limits_{\mathbb{R}}f(z,t)e^{2\pi i\lambda t}dt.
\end{equation}
It can be seen that $f^\lambda\in L^{1}(\mathbb{C}^n)$. For $f\in L^{1}(\mathbb{C}^n)$, the operator $W_\lambda(f)$ on $L^{2}(\mathbb{R}^n)$, is defined as \[W_\lambda(f)=\int\limits_{\mathbb{C}^n}f(z)\pi_\lambda(z,0)dz.\] Clearly, there is a relation between group Fourier transform and $W_\lambda$ given by
\begin{equation}\label{E:Reln}
\hat{f}(\lambda)=W_\lambda(f^\lambda).
\end{equation}
Moreover, $W_\lambda(f)$ is an integral operator on $L^{2}(\mathbb{R}^n)$ with kernel $K^\lambda_f$ given by
\begin{equation}\label{E:Kdef}
K^\lambda_f(\xi,\eta)=\int\limits_{\mathbb{R}^n}f(x,\eta-\xi)e^{\pi i\lambda x\cdot(\xi+\eta)}dx.
\end{equation} 
In particular when $\lambda=1,\,W_\lambda(f)$ is denoted by $W(f)$ which is called the Weyl transform of $f$ and the associated kernel is denoted by $K_f$.
\par
As in the case of the Euclidean Fourier transform, the definitions of $W_\lambda$ and the group Fourier transform $\hat{f}$ can be extended to functions in $L^{2}(\mathbb{C}^n)$ and $L^{2}(\mathbb{H}^n)$ respectively through the density argument. In fact, for $f\in L^{2}(\mathbb{C}^n),\,W_\lambda(f)$ is a Hilbert-Schmidt operator on $L^{2}(\mathbb{R}^n)$ which satisfies
\[\|W_\lambda(f)\|_{\mathcal{B}_2}=\|K^\lambda_f\|_{L^{2}(\mathbb{C}^n)}=\tfrac{1}{|\lambda|^{n/2}}\|f\|_{L^{2}(\mathbb{C}^n)},\]
where $\mathcal{B}_2=\mathcal{B}_2(L^{2}(\mathbb{R}^n))$ denotes the class of Hilbert-Schmidt operators on $L^{2}(\mathbb{R}^n)$. In other words, for $f,g\in L^{2}(\mathbb{C}^n)$,
\begin{equation}\label{E:Wnf}
\langle W_\lambda(f),W_\lambda(g)\rangle_{\mathcal{B}_2}=\langle K^\lambda_f,K^\lambda_g\rangle_{L^{2}(\mathbb{C}^n)}=\tfrac{1}{|\lambda|^n}\langle f,g\rangle_{L^{2}(\mathbb{C}^n)}.
\end{equation}
Furthermore, the group Fourier transform satisfies the Plancherel formula 
\[\|\hat{f}\|_{L^{2}(\mathbb{R}^\ast,\mathcal{B}_2;d\mu)}=\|f\|_{L^{2}(\mathbb{H}^n)},\]
where $L^{2}(\mathbb{R}^\ast,\mathcal{B}_2;d\mu)$ stands for the space of functions on $\mathbb{R}^\ast$ taking values in $\mathcal{B}_2$ and square integrable with respect to the measure $d\mu(\lambda)=|\lambda|^n d\lambda.$ Equivalently, we have
\begin{equation}
\langle\hat{f},\hat{g}\rangle_{L^{2}(\mathbb{R}^\ast,\mathcal{B}_2;d\mu)}=\int\limits_{\mathbb{R}}\langle \hat{f}(\lambda),\hat{g}(\lambda)\rangle_{\mathcal{B}_2}|\lambda|^n d\lambda=\langle f,g\rangle_{L^{2}(\mathbb{H}^n)}.
\end{equation} 
Then, it follows from (\ref{E:Reln}) and (\ref{E:Wnf}) that
\begin{equation}\label{E:Knf}
\langle f,g\rangle_{L^{2}(\mathbb{H}^n)}=\int\limits_{\mathbb{R}}\langle K^\lambda_{f^\lambda},K^\lambda_{g^\lambda}\rangle_{L^{2}(\mathbb{C}^n)}|\lambda|^n d\lambda.
\end{equation}
For a further study on Heisenberg group, we refer to \cite{Foll} and \cite{Thanga}.
\par
We organize the paper as follows. We discuss the orthonormality of the wavelet system on the Heisenberg group in section 2 and that of the twisted wavelet system on $\mathbb{C}$ in section 3.

\section{Orthonormality of wavelet system on $\mathbb{H}$}
Using the translation and the nonisotropic dilation defined in section 1, we consider the wavelet system $\{\delta_{2^j}L_{(k,l,m)}\psi:k,l,m,j\in\mathbb{Z}\}$ given by 
\begin{equation}
\delta_{2^j}L_{(k,l,m)}\psi(x,y,t)=2^{2j}\psi\left(2^j x-k,2^j y-l,2^{2j}t-m+\tfrac{1}{2}2^j(y\cdot k-x\cdot l)\right)
\end{equation}
for $\psi\in L^2(\mathbb{H}).$
\begin{definition}[\cite{RaS1}]\label{D:G}
For $\psi\in L^2(\mathbb{H}),\,(k,l)\in\mathbb{Z}^2$, we define the function $G^\psi_{k,l}$ as
\begin{align*}
G^\psi_{k,l}(\lambda)=\sum\limits_{r\in\mathbb{Z}}\sum\limits_{s\in\mathbb{Z}}\int\limits_0^1 \int\limits_{\mathbb{R}}K^{\lambda+r}_{\psi^{\lambda+r}}&(\xi+s,\eta)\overline{K^{\lambda+r}_{\psi^{\lambda+r}}(\xi+s+l,\eta)}\\
&\times e^{-\pi i(\lambda+r)k(2\xi+l)}e^{-2\pi i\lambda ks}d\eta d\xi|\lambda+r|,\,\lambda\in(0,1].
\end{align*}
\end{definition}

The following theorem is well known.
\begin{theorem}[\cite{RaS1}]\label{T:transortho}
If $\psi\in L^2(\mathbb{H})$, then $\{L_{(k,l,m)}\psi:(k,l,m)\in\mathbb{Z}^3\}$ is an orthonormal system in $L^2(\mathbb{H})$ if and only if the following conditions hold.
\begin{align*}
(i) &G^\psi_{0,0}(\lambda)=1\quad\text{a.e. }\lambda\in(0,1]\quad\text{and}\\
(ii)&G^\psi_{k,l}(\lambda)=0\quad\text{a.e. }\lambda\in(0,1],\text{ for all }(k,l)\neq(0,0)\text{ in }\mathbb{Z}^2,
\end{align*}
where $G^\psi_{k,l}$ is as in Definition \ref{D:G}.
\end{theorem}
\noindent See also \cite{Azit}.
\par
Now, our main result provides the necessary and sufficient conditions for the orthonormality of the system $\{\delta_{2^j}L_{(k,l,m)}\psi:k,l,m,j\in\mathbb{Z}\}$ on $\mathbb{H}$ which is stated as follows.
\begin{theorem}\label{T:orthoHei}
Let $\psi\in L^2(\mathbb{H})$. For $j_1,j_2,k_1,k_2,l_1,l_2\in\mathbb{Z},\,\lambda\in(0,1] $, let
\begin{equation}\label{E:Fdef}
\begin{split}
&F^\psi_{j_1,j_2,k_1,k_2,l_1,l_2}(\lambda)\\
&\,=\sum\limits_{r\in\mathbb{Z}}\sum\limits_{s\in\mathbb{Z}}\int\limits_0^1 \int\limits_{\mathbb{R}}K^{2^{2(j_2-j_1)}(\lambda+r)}_{\psi^{2^{2(j_2-j_1)}(\lambda+r)}}(2^{j_1}(\xi+s)+l_1,2^{j_1}\eta)\overline{K^{\lambda+r}_{\psi^{\lambda+r}}(2^{j_2}(\xi+s)+l_2,2^{j_2}\eta)}\\
&\qquad\qquad\quad\;\,\times e^{\pi i2^{2(j_2-j_1)}(\lambda+r)k_1(2^{j_1+1}(\xi+s)+l_1)}e^{-\pi i(\lambda+r)k_2(2^{j_2+1}(\xi+s)+l_2)}d\eta d\xi|\lambda+r|. 
\end{split}
\end{equation}

The wavelet system $\{\delta_{2^j}L_{(k,l,m)}\psi:k,l,m,j\in\mathbb{Z}\}$ is orthonormal in $L^2(\mathbb{H})$ if and only if the following conditions hold.
\begin{align*}
(i) &G^\psi_{0,0}(\lambda)=1\quad\text{a.e. }\lambda\in(0,1],\\
(ii)&G^\psi_{k,l}(\lambda)=0\quad\text{a.e. }\lambda\in(0,1],\text{ for all }(k,l)\neq(0,0)\text{ in }\mathbb{Z}^2,\\
(iii)&F^\psi_{j_1,j_2,k_1,k_2,l_1,l_2}(\lambda)=0\quad\text{a.e. }\lambda\in(0,1],\text{ for }j_2>j_1,k_1,k_2,l_1,l_2\text{ in }\mathbb{Z},
\end{align*}
where $G^\psi_{k,l}$ is as in Definition \ref{D:G}.
\end{theorem}
In order to prove Theorem \ref{T:orthoHei}, we shall prove the following lemmas. At first, in view of (\ref{E:Knf}), we need to determine the inverse Fourier transform of $\delta_{2^j}L_{(k,l,m)}\psi$ with respect to $t$ variable. In other words, we need to determine $(\delta_{2^j}L_{(k,l,m)}\psi)^\lambda$ and also the associated kernel $K^\lambda_{(\delta_{2^j}L_{(k,l,m)}\psi)^\lambda}$. In Lemma \ref{L:waveNtwis}, we shall express $(\delta_{2^j}L_{(k,l,m)}\psi)^\lambda$ in terms of the $\lambda$-twisted translation $(T^t_{(k,l)})^\lambda$ and dilation $\mathcal{D}_{2^j}$ on $\mathbb{R}^2$, where $(T^t_{(k,l)})^\lambda$ and $\mathcal{D}_{2^j}$ are defined below.

\begin{definition}\label{D:twis}
For $(k,l)\in\mathbb{Z}^2$ and $j\in\mathbb{Z}$, the $\lambda$-twisted translation $(T^t_{(k,l)})^\lambda$ and the dilation $\mathcal{D}_{2^j}$ are defined as follows.
\begin{align*}
(T^t_{(k,l)})^\lambda\varphi(x,y)&=e^{\pi i\lambda(x\cdot l-y\cdot k)}\varphi(x-k,y-l)\qquad\text{ and}\\
 \mathcal{D}_{2^j}\varphi(x,y)&=2^j\varphi(2^jx,2^jy),\qquad\quad \text{for }\varphi\in L^2(\mathbb{R}^2).
\end{align*}
\end{definition}

\begin{lemma}\label{L:waveNtwis}
For $\psi\in L^2(\mathbb{H}),\, (k,l,m)\in\mathbb{Z}^3,\, j\in\mathbb{Z}$, the inverse Fourier transform of $\delta_{2^j}L_{(k,l,m)}\psi$ with respect to $t$ variable satisfies \[(\delta_{2^j}L_{(k,l,m)}\psi)^\lambda=2^{-j}e^{2\pi i\lambda2^{-2j}m}\mathcal{D}_{2^j}(T^t_{(k,l)})^{\lambda2^{-2j}}\psi^{\lambda2^{-2j}},\]
where $\psi^{\lambda2^{-2j}}$ is the inverse Fourier transform of $\psi$ in the $t$ variable given by $\psi^{\lambda 2^{-2j}}(z)=\int_{\mathbb{R}}\psi(z,t)e^{2\pi i\lambda 2^{-2j}t}dt.$
\end{lemma}
\begin{proof}
Consider,
\begin{align*}
(\delta_{2^j}L_{(k,l,m)}&\psi)^\lambda(x,y)\\
&=\int\limits_\mathbb{R}2^{2j}\psi\left(2^j x-k,2^j y-l,2^{2j}t-m+\tfrac{1}{2}2^j(y\cdot k-x\cdot l)\right)e^{2\pi i\lambda t}dt.
\end{align*}
Now, by applying the change of variable $s=2^{2j}t-m+\tfrac{1}{2}2^j(y\cdot k-x\cdot l)$, we get,
\begin{align*}
(\delta_{2^j}L_{(k,l,m)}&\psi)^\lambda(x,y)\\
&=e^{2\pi i\lambda 2^{-2j}m}e^{2^j\pi i\lambda2^{-2j}(x\cdot l-y\cdot k)}\int\limits_\mathbb{R}\psi(2^j x-k,2^j y-l,s)e^{2\pi i\lambda2^{-2j}s}ds\\
&=e^{2\pi i\lambda 2^{-2j}m}e^{\pi i\lambda2^{-2j}(2^jx\cdot l-2^jy\cdot k)}\psi^{\lambda2^{-2j}}(2^j x-k,2^j y-l)\\
&=2^{-j}e^{2\pi i\lambda 2^{-2j}m}\mathcal{D}_{2^j}(T^t_{(k,l)})^{\lambda2^{-2j}}\psi^{\lambda2^{-2j}}(x,y).
\end{align*}
\end{proof}
\begin{lemma}\label{L:KdilnKwave}
Let $\varphi\in L^2(\mathbb{R}^2)$. Then the kernel of $W_\lambda(\mathcal{D}_{2^j}\varphi)$ and hence that of $W_\lambda(\mathcal{D}_{2^j}(T^t_{(k,l)})^{\lambda2^{-2j}}\varphi)$, denoted by $K^\lambda_{\mathcal{D}_{2^j}\varphi}$ and $K^\lambda_{\mathcal{D}_{2^j}(T^t_{(k,l)})^{\lambda2^{-2j}}\varphi}$ respectively, satisfy the following relations.
\begin{align}
K^\lambda_{\mathcal{D}_{2^j}\varphi}(\xi,\eta)&=K^{\lambda2^{-2j}}_\varphi(2^j\xi,2^j\eta),\label{E:Kdil}\\
K^\lambda_{\mathcal{D}_{2^j}(T^t_{(k,l)})^{\lambda2^{-2j}}\varphi}(\xi,\eta)&=e^{\pi i\lambda2^{-2j}k(2^{j+1}\xi+l)}K^{\lambda2^{-2j}}_\varphi(2^j\xi+l,2^j\eta).\label{Ktwiswave}
\end{align}
\end{lemma}
\begin{proof}
From (\ref{E:Kdef}) and Definition \ref{D:twis}, we have 
\[K^\lambda_{\mathcal{D}_{2^j}\varphi}(\xi,\eta)=\int\limits_{\mathbb{R}}2^j\varphi(2^jx,2^j(\eta-\xi))e^{\pi i\lambda x\cdot(\xi+\eta)}dx.\]
Applying the change of variable $u=2^j x$, we get 
\[K^\lambda_{\mathcal{D}_{2^j}\varphi}(\xi,\eta)=\int\limits_{\mathbb{R}}\varphi(u,2^j\eta-2^j\xi)e^{\pi i\lambda 2^{-2j}u\cdot(2^j\xi+2^j\eta)}du.\]
The right hand side of the above equation is nothing but $K^{\lambda2^{-2j}}_\varphi(2^j\xi,2^j\eta)$, thus proving (\ref{E:Kdil}).
\par
Now, making use of (\ref{E:Kdil}), we see that
\begin{align*}
&K^\lambda_{\mathcal{D}_{2^j}(T^t_{(k,l)})^{\lambda2^{-2j}}\varphi}(\xi,\eta)=K^{\lambda 2^{-2j}}_{(T^t_{(k,l)})^{\lambda2^{-2j}}\varphi}(2^j\xi,2^j\eta)\\
&\qquad=\int\limits_\mathbb{R}e^{\pi i\lambda2^{-2j}(x\cdot l-(2^j\eta-2^j\xi)\cdot k)}\varphi(x-k,2^j\eta-2^j\xi-l)e^{\pi i\lambda2^{-2j}x\cdot(2^j\xi+2^j\eta)}dx\\
&\qquad=e^{\pi i\lambda2^{-2j}k(2^{j+1}\xi+l)}\int\limits_\mathbb{R}\varphi(u,2^j\eta-(2^j\xi+l))e^{\pi i\lambda 2^{-2j}u\cdot(2^j\xi+2^j\eta+l)}du\\
&\qquad=e^{\pi i\lambda2^{-2j}k(2^{j+1}\xi+l)}K^{\lambda2^{-2j}}_\varphi(2^j\xi+l,2^j\eta),
\end{align*}
thereby proving (\ref{Ktwiswave}).
\end{proof}

We observe from Lemma \ref{L:waveNtwis} that there is a connection between wavelet system on the Heisenberg group and a wavelet system on $\mathbb{C}$, namely the twisted wavelet system $\{\mathcal{D}_{2^j}(T^t_{(k,l)})^{2^{-2j}}\varphi:k,l,j\in\mathbb{Z}\}$ for $\varphi\in L^2(\mathbb{C})$. Thus it becomes a natural question to investigate the necessary and sufficient condition for the orthonormality of the twisted wavelet system on $\mathbb{C}$. This problem is studied in the forthcoming section.

\section{Orthonormality of twisted wavelet system on $\mathbb{C}$}
Before stating the main result of this section, we shall consider the following lemma.
\begin{lemma}\label{L:Ktwistrans}
Let $\varphi\in L^2(\mathbb{R}^2)$. Then the kernel of the Weyl transform of $\mathcal{D}_{2^j}\varphi$, $(T^t_{(k,l)})^{2^{-2j}}\varphi$ and $\mathcal{D}_{2^j}(T^t_{(k,l)})^{2^{-2j}}\varphi$ respectively satisfy the following relations. 
\begin{align*}
(i&)\, K_{\mathcal{D}_{2^j}\varphi}(\xi,\eta)=K_\varphi^{2^{-2j}}(2^j\xi,2^j\eta),\quad j\in\mathbb{Z}, \\
(ii&)\, K_{(T^t_{(k,l)})^{2^{-2j}}\varphi}(\xi,\eta)\\
&\qquad\;=e^{\pi i2^{-2j}kl}e^{\pi ik(1+2^{-2j})\xi}e^{\pi ik(1-2^{-2j})\eta}K_{e((\frac{2^{-2j}-1}{2})l,0)\varphi}(\xi+l,\eta),\\
&\; where,\, for\,(a,b)\in\mathbb{R}^{2},\, the\, operator\, e(a,b)\, on\,L^{2}(\mathbb{R}^{2})\, is\, given\, by \\ 
&\; (e(a,b)\varphi)(x,y)=e^{2\pi i(ax+by)}\varphi(x,y),\\
(iii&\,) K_{\mathcal{D}_{2^j}(T^t_{(k,l)})^{2^{-2j}}\varphi}(\xi,\eta)=e^{\pi i 2^{-2j}k(2^{j+1}\xi+l)}K_\varphi^{2^{-2j}}(2^j\xi+l,2^j\eta).
\end{align*}
\end{lemma}

In \cite{RaS}, Radha and Adhikari proved the following
\begin{theorem}\label{T:twistransortho}
Let $\varphi\in L^2(\mathbb{R}^2)$. Then $\{T^t_{(k,l)}\varphi:k,l\in\mathbb{Z}\}$ is an orthonormal system in $L^2(\mathbb{R}^2)$ if and only if 
\begin{align*}
(i)&\sum\limits_{m\in\mathbb{Z}}\int\limits_{\mathbb{R}}|K_\varphi(\xi+m,\eta)|^2d\eta=1\quad\text{  a.e. }\xi\in[0,1]\quad\text{ and}\\
(ii)&\sum\limits_{m\in\mathbb{Z}}\int\limits_{\mathbb{R}}K_\varphi(\xi+m,\eta)\overline{K_\varphi(\xi+m+l,\eta)}d\eta=0\quad\text{ a.e. }\xi\in[0,1],\,\forall\,l\in\mathbb{Z}\setminus\{0\}.
\end{align*}
\end{theorem}
We shall now state a similar result in the context of twisted wavelet system.
\begin{theorem}\label{T:orthoTwis}
Let $\varphi\in L^2(\mathbb{R}^2)$. Let $\xi\in[0,1]$. For $j_1,j_2,l_1,l_2\in\mathbb{Z},$ let 
\begin{equation}\label{E:Pdef}
\begin{split}
&P^\varphi_{j_1,j_2,l_1,l_2}(\xi)\\
&\qquad=\sum\limits_{m\in\mathbb{Z}}\int\limits_{\mathbb{R}}K_\varphi^{2^{-2j_1}}(2^{j_1+j_2}(\xi+m)+l_1,2^{j_1}\eta)
\overline{K_\varphi^{2^{-2j_2}}(2^{2j_2}(\xi+m)+l_2,2^{j_2}\eta)}d\eta.
\end{split}
\end{equation}
\begin{align}
\intertext{For $j,l,m\in\mathbb{Z},\, y\in\mathbb{R},$ let}
&S^{\varphi}_{j,l}(\xi,m,y)=K_{e((\frac{2^{-2j}-1}{2})l,0)\varphi}
\left(\frac{2(\xi+m-y)}{1+2^{-2j}}+l,\frac{2y}{1-2^{-2j}}\right).\label{E:Sdef}\\
\intertext{For $j,l\in\mathbb{Z}$, let}
&R^\varphi_{j,l}(\xi)=\sum\limits_{m\in\mathbb{Z}}\int\limits_{\mathbb{R}}|S^{\varphi}_{j,l}(\xi,m,y)|^2dy.\label{E:Rdef}\\
\intertext{For $j,l_1,l_2\in\mathbb{Z}$, let}
&Q^\varphi_{j,l_1,l_2}(\xi)=\sum\limits_{m\in\mathbb{Z}}\int\limits_{\mathbb{R}}S^{\varphi}_{j,l_1}(\xi,m,y)\overline{S^{\varphi}_{j,l_2}(\xi,m,y)}dy.\label{E:Qdef}
\end{align}
 Then the twisted wavelet system $\{\mathcal{D}_{2^j}(T^t_{(k,l)})^{2^{-2j}}\varphi:k,l,j\in\mathbb{Z}\}$ is orthonormal in $L^2(\mathbb{R}^2)$ if and only if
\begin{align*}
(i&)\,\sum\limits_{m\in\mathbb{Z}}\int\limits_{\mathbb{R}}|K_\varphi(\xi+m,\eta)|^2d\eta=1\quad\text{ a.e. }\xi\in[0,1],\\
(ii&)\,\sum\limits_{m\in\mathbb{Z}}\int\limits_{\mathbb{R}}K_\varphi(\xi+m,\eta)\overline{K_\varphi(\xi+m+l,\eta)}d\eta=0\quad\text{ a.e. }\xi\in[0,1],\,\forall\,l\in\mathbb{Z}\setminus\{0\}, \\
(iii&)\, P^\varphi_{j_1,j_2,l_1,l_2}(\xi)=0\quad\text{ a.e. }\xi\in[0,1],\,\text{ for }j_2>j_1,\, l_1,l_2\in\mathbb{Z},\\
(iv&)\, Q^\varphi_{j,l_1,l_2}(\xi)=0\quad\text{ a.e. }\xi\in[0,1],\,\text{ for }j\in\mathbb{Z}\setminus\{0\},\, l_1\neq l_2\text{ in }\mathbb{Z},\\
(v&)\, R^\varphi_{j,l}(\xi)=\tfrac{|1-2^{-4j}|}{4}\quad\text{ a.e. }\xi\in[0,1],\,\text{ for }j\in\mathbb{Z}\setminus\{0\},\, l\in\mathbb{Z}.
\end{align*}
\end{theorem}

\bibliographystyle{amsplain}
\bibliography{references}

\providecommand{\bysame}{\leavevmode\hbox to3em{\hrulefill}\thinspace}
\providecommand{\MR}{\relax\ifhmode\unskip\space\fi MR }
\providecommand{\MRhref}[2]{%
  \href{http://www.ams.org/mathscinet-getitem?mr=#1}{#2}
}
\providecommand{\href}[2]{#2}
\begin{thebibliography}{10}

\bibitem{Azit}
D.~Barbieri, E.~Hern\'{a}ndez, and A.~Mayeli, \emph{Bracket map for the
  {Heisenberg} group and the characterization of cyclic subspaces}, Appl.
  Comput. Harmon. Anal. \textbf{37} (2014), 218--234.

\bibitem{Bow}
M.~Bownik, \emph{The structure of shift-invariant subspaces of {$
  L^2(\mathbb{R}^n)$}}, J. Funct. Anal. \textbf{176} (2000), 282--309.

\bibitem{Bow1}
\bysame, \emph{The structure of shift-modulation invariant spaces: The rational
  case}, J. Funct. Anal. \textbf{244} (2007), 172--219.

\bibitem{Cab}
C.~Cabrelli and V.~Paternostro, \emph{Shift-invariant spaces on {LCA} groups},
  J. Funct. Anal. \textbf{258} (2010), 2034--2059.

\bibitem{Cab1}
\bysame, \emph{Shift-modulation invariant spaces on {LCA} groups}, Studia Math.
  \textbf{211} (2012), 1--19.

\bibitem{Chris}
O.~Christensen, \emph{Frames and bases: An introductory course},
  Birkh\"{a}user, Boston, 2008.

\bibitem{Foll}
G.~B. Folland, \emph{Harmonic analysis in phase space}, Princeton University
  Press, Princeton, New Jersey, 1989.

\bibitem{HW}
E.~Hern\'{a}ndez and G.~Weiss, \emph{A first course on wavelets}, CRC Press,
  Boca Raton, 1996.

\bibitem{Kamy}
R.A. Kamyabi~Gol and R.~Raisi~Tousi, \emph{The structure of shift invariant
  spaces on a locally compact abelian group}, J. Math. Anal. Appl. \textbf{340}
  (2008), 219--225.

\bibitem{RaS}
R.~Radha and Saswata Adhikari, \emph{Frames and {Riesz} bases of twisted
  shift-invariant spaces in {$L^2(\mathbb{R}^{2n})$}}, J. Math. Anal. Appl.
  \textbf{434} (2016), 1442--1461.

\bibitem{RaS1}
\bysame, \emph{Shift-invariant spaces with countably many mutually orthogonal
  generators on the {Heisenberg} group}, arXiv:1711.06902v2 [math.FA], 2017.

\bibitem{RaSh}
R.~Radha and N.~Shravan~Kumar, \emph{Shift-invariant subspaces on compact
  groups}, Bull. Sci. Math. \textbf{137 (4)} (2013), 485--497.

\bibitem{Thanga}
S.~Thangavelu, \emph{Harmonic analysis on the {Heisenberg} group},
  Birkh\"{a}user, Boston, 1998.

\end{thebibliography}

\end{document}